\documentclass[runningheads,a4paper,envcountsame]{llncs}

\usepackage{amsfonts}
\usepackage{amssymb}
\usepackage{mathtools}
\usepackage[all,cmtip]{xy}
\usepackage[bookmarks=false]{hyperref} 
\usepackage{xcolor}
\usepackage{todonotes}

\renewcommand{\phi}{\varphi}

\title{Almost sure OTM-realizability}
\author{Merlin Carl}
\date{August 2023}
\institute{Institut f\"ur Mathematik, Europa-Universit\"at Flensburg}

\begin{document}

\maketitle

\begin{abstract}
Combining the approaches made in works with Galeotti and Passmann, we define and study a notion of ``almost sure'' realizability with parameter-free ordinal Turing machines (OTMs). In particular, we show that, in contrast to the classical case, almost sure realizability differs from plain realizability, while closure under intuitionistic predicate logic and realizability of Kripke-Platek set theory continue to hold.
\end{abstract}

\section{Introduction}

Kleene realizability is one of the most prominent interpretations for intuitionistic logic. However, by its reliance on Turing computability, it is limited to finite objects and procedures and not adapted to notions of generalized or ``relaxed'' effectivity, as described, e.g., in Hodges \cite{Hodges}. Therefore, a number of variants of Kleene realizability has been defined, among them one which is based on computability with Koepke's ordinal Turing machines (OTMs) \cite{Koepke} rather than Turing machines (see \cite{CGP-OTM}). On the other hand, in \cite{CGP}, a ``randomized'' version of realizability was investigated, in which realizers are merely supposed to work with high probability relative to a random oracle. One of the results of \cite{CGP} (Theorem $25$) is that, due to Sacks' theorem (see, e.g., \cite{DH}, Corollary 8.12.2), requiring random realizers to work with probability $1$ leads back to classical realizability. 

In this paper, we combine the two variants by considering a randomized version of OTM-realizability. Although the analogue of Sacks' theorem for OTMs is consistent with ZFC by (\cite{CS}, sections 2.1 and 2.3), it turns out that randomized OTM-realizability with probability $1$ is still different from ``plain'' OTM-realizability. This yields the notion of ``almost sure OTM-realizability'' (as-OTM-realizability). 

We find that that as-OTM-realizability enjoys similar features to the concept of OTM-realizability studied in \cite{CGP}: In particular, we will show that as-OTM-realizability is sound for the deduction rules of intuitionistic predicate calculus and that all axioms of Kripke-Platek set theory are as-OTM-realizable, while some axioms of ZFC that go beyond KP are not. Thus, as-OTM-realizability at least captures Lubarsky's intuitionistic KP (IKP), see \cite{Lubarsky}, \cite{IP}.

There does not seem to be a meaningful analogue of Lebesgue measurability, and thus of randomness, for domains that go beyond the real numbers, see, e.g. chapter $5$ of \cite{Wontner}.\footnote{We thank Lorenzo Galeotti for pointing out this reference to us.} Thus, our random oracles will be real numbers. Since the random objects should be on par with the objects in the realm under consideration, we will restrict ourselves in this paper to statements about $H_{\omega_{1}}$, the set of heriditarily countable sets, as precisely these sets can be encoded as real numbers.\footnote{Whether or not there is an interesting theory of OTM-realizability relative to random oracles on arbitrary sets is a question that we postpone to future work.} Thus, unless stated otherwise, quantifiers will range over $H_{\omega_{1}}$ below. We recall that $(H_{\omega_{1}})^{L}=L_{\omega_{1}^{L}}$. 

Most results in this paper are analogous to results obtained in joint work with Robert Passmann and Lorenzo Galeotti for (non-randomized) OTM-realizability \cite{CGP-OTM} and randomized (classical) realizability \cite{CGP}, although our definition of randomized OTM-realizability differs in several respect from what a straightforward combination of these two approaches would look like.

\section{Definitions and basic results}

For the definition of OTMs, we refer to \cite{Koepke}. 
We write $P^{x}(y)\downarrow=z$ to indicate that $P$, when run on input $x$ in the oracle $y$, halts with output $z$; if $x$ or $y$ are omitted, we mean the computation is started on the empty tape or the empty oracle; if $z$ is omitted, we mean that the computation stops without specifying the output. We fix a natural enumeration $(P_{i}:i\in\omega)$ of OTM-programs. Moreover, we fix a natural enumeration $(\phi_{i}:i\in\omega)$ of $\in$-formulas and denote by $\lceil\phi\rceil$ the index of $\phi$ in this enumeration.

For $x$ and $y$ real numbers, we define $x\oplus y$ as $\{2i:i\in x\}\cup\{2j+1:j\in y\}$; more generally, for $x_{0},...,x_{n}$ real numbers, we define $\bigoplus_{i=0}^{0}x_{i}:=x_{0}$ and $\bigoplus_{i=0}^{n}x_{i}:=x_{0}\oplus\bigoplus_{i=1}^{n}x_{i}$ for $n>2$.

In order to pass sets as inputs to OTMs, these need to be encoded as sets of ordinals. This can be done as described in \cite{CarlBook}, Definition 2.3.18. However, since what we are after is a notion of realizability on sets, not set codes, we will require our programs to behave independently of the specific code given to them:

\begin{definition}
    An OTM-program $P$ is \textit{coding-stable} if and only if, for all sets $a_{1},...,a_{n}$, and all sets of ordinals $c_{1}^{0},...,c_{n}^{0}$ and $c_{1}^{1},...,c_{n}^{1}$ coding $a_{1},...,a_{n}$, 
    $P^{c_{1}^{0},...,c_{n}^{0}}$ halts if and only if $P^{c_{1}^{1},...,c_{n}^{1}}$ does and moreover, if $c^{0}$ is the output of the first computation and $c^{1}$ is the output of the second computation, then $c^{0}$ and $c^{1}$ encode the same set. 

    We say that $P$ is \textit{safe} if and only if, for all sets of ordinals $x$ and $y$, $P^{x}(y)$ halts.
\end{definition}

One of the technical advantages of coding stability is that we can abuse our notation and confuse sets with their codes in the discussion. Thus, when $P$ is coding-stable and $x$, $y$ are sets, we write $P^{x}(y)$ for the set coded by the output of $P^{x}$ run with $y$ on the input tape; in particular, this means that this computation halts.

The definition of as-OTM-realizability is adapted from the definition of ``big realizability'' in \cite{CGP} to OTMs.

\begin{definition}{\label{main def}}
    For $X\subseteq\mathfrak{P}(\omega)$, denote by $\mu(X)$ the Lebesgue measure of $X$, if it exists. Let $\mathcal{B}$ denote the set of $X\subseteq\mathfrak{P}(\omega)$ such that $\mu(X)=1$. 

    We will define what it means for a pair $(P,\vec{p})=(P,(p_{1},...,p_{n}))$ to as-OTM-realize an $\in$-statement $\phi(\vec{x},\vec{a})$, where $\vec{x}$ is a list of free variables and $\vec{a}$ is a list of set parameters. $P$ will always be a safe and coding-stable OTM-program, while $p_{1},...,p_{n}$ will be a tuple of (codes for) heriditarily countable sets; for simplicity, we will assume that $p_{1},...,p_{n}$ are always real numbers.\footnote{In order not to clutter our notation with iterated indices, the tuple of parameters will always consist of $n$ entries, although, of course, the number of parameters is variable.}

    If $\phi(\vec{x},\vec{a})$ contains free variables, then $(P,\vec{p})\Vdash_{\text{as}}^{\text{OTM}}\phi(\vec{x},\vec{a})$ if and only if $(P,\vec{p})\Vdash_{\text{as}}^{\text{OTM}}\forall{\vec{x}}\phi(\vec{a})$, where the latter is defined as below. 

     We can thus focus on the case that $\phi$ does not contain free variables. 

    \begin{enumerate}

    \item If $\phi$ is atomic, then $(P,\vec{p})\Vdash_{\text{as}}^{\text{OTM}}\phi(\vec{a})$ if and only if $\phi(\vec{a})$ is true.

    \item If $\phi$ is $\psi_{0}\wedge\psi_{1}$, then $(P,\vec{p})\Vdash_{\text{as}}^{\text{OTM}}\phi(\vec{a})$ if and only if there is a set $\mathcal{O}\in\mathcal{B}$ such that, for all $x\in\mathcal{O}$, $P^{\vec{p},\vec{a},x}(i)\Vdash_{\text{as}}^{\text{OTM}}\psi_{i}(\vec{a})$ for $i\in\{0,1\}$. 

    \item If $\phi$ is $\psi_{0}\vee\psi_{1}$, then $(P,\vec{p})\Vdash_{\text{as}}^{\text{OTM}}\phi(\vec{a})$ if and only if there is a set $\mathcal{O}\in\mathcal{B}$ such that, for all $x\in \mathcal{O}$, $P^{\vec{p},\vec{a},x}(0)$ terminates with output $i\in\{0,1\}$ and $P^{\vec{p},\vec{a},x}(1)\Vdash_{\text{as}}^{\text{OTM}}\psi_{i}(\vec{a})$.\footnote{Here, we deviate conceptually from the definition in \cite{CGP} by allowing that, for different oracles, the disjuncts to be realized may also differ.}

    \item If $\phi$ is $\psi_{0}\rightarrow\psi_{1}$, then $(P,\vec{p})\vdash_{\text{as}}^{\text{OTM}}\phi(\vec{a})$ if and only if, for every as-OTM-realizer $s$ of $\psi_{0}(\vec{a})$, there is a set $\mathcal{O}\in\mathcal{B}$ such that, for all $x\in \mathcal{O}$, $P^{\vec{p},\vec{a},x}(s)\Vdash_{\text{as}}^{\text{OTM}}\psi_{1}(\vec{a})$.\footnote{Intuitively, the random real helping with the realization is only chosen after the evidence for the premise of the implication was obtained. This again deviates from the definition in \cite{CGP}.}

    \item If $\phi$ is $\neg\psi$, then $(P,\vec{p})\Vdash_{\text{as}}^{\text{OTM}}\phi(\vec{a})$ if and only if $(P,\vec{p})\Vdash\phi\rightarrow(1=0)$. 

    \item If $\phi$ is $\exists{y}\psi(y)$, then $(P,\vec{p})\Vdash_{\text{as}}^{\text{OTM}}\phi(\vec{a})$ if and only if there is $\mathcal{O}\in\mathcal{B}$ such that, for all $x\in\mathcal{O}$, we have: $$P^{\vec{p},\vec{a},x}(0)\Vdash_{\text{as}}^{\text{OTM}}\psi(P^{\vec{p},\vec{a},x}(1),\vec{a}).$$

    \item If $\phi$ is $\forall{y}\psi$, then $(P,\vec{p})\Vdash_{\text{as}}^{\text{OTM}}\phi(\vec{a})$ if and only if, for every set $b$, there is $\mathcal{O}_{b}\in\mathcal{B}$ such that, for all $x\in\mathcal{O}_{b}$, we have that $P^{\vec{p},\vec{a},b,x}\Vdash\psi(b)$.

    \item Bounded quantifiers will be interpreted in the usual way -- i.e., $\exists{x\in a}\phi(x)$ as $\exists{x}(x\in a \wedge \phi(x))$ and $\forall{x\in a}\phi(x)$ as $\forall{x}(x\in a\rightarrow\phi(x))$ -- and interpreted accordingly using the above clauses.
    \end{enumerate}

    In all clauses expect (1), some input $\xi$ is given to the program. We call this a \textit{relevant input} for $(P,\vec{p})$ and $\phi$ (thus, e.g. in the case of conjunction, the relevant input can be $0$ or $1$). For the sake of uniformity, we make the convention that the relevant input for atomic formulas is $0$.

    In the cases (2), (3) and (6), the set $\mathcal{O}$ is called the \textit{success set} of $(P,\vec{p})$ for $\phi$; in (4), (5) and (7), $\mathcal{O}_{z}$ is called the $z$-\textit{success set} of $(P,\vec{p})$ for $\phi$, denoted $\text{succ}^{z}_{(P,\vec{p})}(\phi)$. Again, in order to allow for a uniform treatment of all cases, we speak of the $\xi$-success set of $(P,\vec{p})$ for $\phi$ in all cases, where $\xi$ is a relevant input. In the cases (2), (3) and (6), $\text{succ}_{(P,\vec{p})}^{\xi}(\phi)$ 
    will thus not depend on $\xi$. By convention, $\text{succ}^{\xi}_{(P,\vec{p})}(\phi)=\mathfrak{P}(\omega)$ for atomic $\phi$. 
    The formula $\phi$ will not be mentioned if it is clear from the context, which is usually the case, in which case will talk about the $\xi$-success set of $(P,\vec{p})$ and drop the argument $\phi$.

    We say that a formula $\phi(\vec{a})$ with parameters $\vec{a}$ is as-OTM-realizable if and only if there is an OTM-program $P$ such that $(P,\vec{a})\Vdash_{\text{as}}^{\text{OTM}}\phi(\vec{a})$.\footnote{Thus, the program is initially only allowed to access those parameters contained in the formula, but not others. Note, however, that, in the clause for implication, the as-OTM-realizers for the antecedent are allowed to use arbitrary parameters.}
   
\end{definition}

Dropping all mentions of oracles in this definition, one obtains the definition of \textit{plain} OTM-realizability, a variant of which was discussed in \cite{CarlNote}, and another variant for infinitary logic in \cite{CGP-OTM}.

\begin{lemma}{\label{definability of realizability}}
For each $\in$-formula $\phi$ and each tuple $\vec{p}$ of real numbers, the following are true:

\begin{enumerate}
    \item The set of as-OTM-realizers of $\phi(\vec{p})$ is analytic in $\vec{p}$; in other words, the statement ``$r\Vdash_{\text{as}}^{\text{OTM}}\phi(\vec{p})$ is $\Sigma^{1}_{n}$ in the parameter $\vec{p}$ for some $n\in\omega$.
    \item For every relevant input $\xi$ and every pair $\tau:=(P,\vec{q})$ with $P$ a safe and coding-stable OTM-program and $\vec{q}$ a parameter, $\text{succ}^{\xi}_{r}(\phi)$ is projective.
\end{enumerate}
\end{lemma}
\begin{proof}
    We prove this by a simultaneous induction on formulas, ignoring the parameter $\vec{p}$. 

    \begin{itemize}
        \item Suppose that $\phi$ is atomic. Then $r\Vdash_{\text{as}}^{\text{OTM}}\phi$ is equivalent to $\phi$, which is $\Sigma_{1}^{0}$; thus, we have (i). Moreover, we have $\xi-\text{succ}_{r}(\phi)=\mathfrak{P}(\omega)$ whenever $r\Vdash_{\text{as}}^{\text{OTM}}\phi$, which is clearly projective.
        \item Suppose that $\phi$ is $(\psi_{0}\wedge\psi_{1})$. 
        Then $r_{\text{as}}^{\text{OTM}}\phi$ is equivalent to the statement ``There exists a set $\mathcal{O}$ of measure $1$ such that, for all $x\in\mathcal{O}$, there are OTM-computations $r^{x}(0)\downarrow=r_{0}$, $r^{x}(1)\downarrow=r_{1}$ satisfying $r_{0}\Vdash_{\text{as}}^{\text{OTM}}\psi_{0}$ and $r_{1}\Vdash_{\text{as}}^{\text{OTM}}\psi_{1}$''. By induction, the statements $r_{0}\Vdash_{\text{as}}^{\text{OTM}}\psi_{0}$ and $r_{1}\Vdash_{\text{as}}^{\text{OTM}}\psi_{1}$ are $\Sigma^{1}_{n}$, for some $n\in\omega$. The existence of halting OTM-computations is also thus expressible, due to the fact that halting OTM-computations on real inputs will be countable, and thus themselves encodeable as real numbers. Finally, stating the existence of a set $\mathcal{O}$ of measure $1$ amounts to stating that there is a null set $\overline{\mathcal{O}}$ such that the claim in question holds for all $x\notin\overline{\mathcal{O}}$; but the existence of a null set can be expressed as in the proof of Lemma \ref{probable halting is sigma2} below by stating the existence of a real number coding an appropriate sequence of interval sequences. Thus, we have (i). To see (ii), let $\tilde{r}$ be the OTM-program that, in the oracle $x$, simulates $r^{x}(i)$, for $i\in\{0,1\}$. Thus, for all $\xi$, we have that $\text{succ}^{\xi}_{r}(\phi)=\text{succ}^{\xi}_{r_{0}}(\psi_{0})\cap\text{succ}^{\xi}_{r_{1}}(\psi_{1})$ is projective as an intersection of two projective sets. 
        \item If $\phi$ is $(\psi_{0}\vee\psi_{1})$, we proceed as in the last case, replacing ``and'' with ``or'' in (i) and ``$\cap$'' with ``$\cup$'' in (ii).

    \end{itemize}
In the remaining cases, the techniques explained in the conjunction case can be straightforwardly applied, in combination with the respective inductive assumptions, to express the respective clauses of the definition in the desired way.
\end{proof}

\begin{lemma}{\label{pi2-truth}}(Cf. \cite{CGP-OTM}, Lemma $21$-$23$)

There is an OTM-program $P_{\text{eff}-\Pi_{2}}$ that, for a true formula $\phi$ of the form $\forall{x}\exists{y}\psi$, where $\psi$ is $\Delta_{0}$, $P_{\text{eff}-\Pi_{2}}(\lceil\phi\rceil)$ outputs a plain OTM-realizer for $\phi$ which is also an as-OTM-realizer for $\phi$. 

Consequently, the following are equivalent for such $\phi$:
    
\begin{enumerate}
    \item $\phi$ is plainly OTM-realizable. 
    \item $\phi$ is as-OTM-realizable.
    \item $\phi$ is true.
\end{enumerate}

\end{lemma}

We observe that as-OTM-realizability differs from plain OTM-realizability:

\begin{theorem}{\label{difference}}
    There is an $\in$-sentence $\phi$ such that $\not\Vdash_{\text{OTM}}\phi$, but $\Vdash_{\text{as-OTM}}\phi$.
\end{theorem}
\begin{proof}
The proof is an adaptation of the example for the case of randomized Turing-computability in \cite{CGP}, Theorem $14$, which somewhat simplifies in the transfinite setting.\footnote{Also note that the example in \cite{CGP} only guarantees a positive success probability, but not probability $1$.} 

Let $\phi$ be the sentence stating the existence of an OTM-incomputable real number, i.e., ``There is a real number $x$ such that, for all natural numbers $k$ and all ordinals $\alpha$, the following holds: $P_{k}$ does not halt in $\alpha$ many steps with output $x$''.
        
It is clear that $\phi$ is not plainly OTM-realizable; for to realize $\phi$, one would have to compute a real number $x$ witnessing the existential statement on an OTM. However, if $P_{k}$ was a program halting with output $x$ after, say, $\alpha$ many steps, this would contradict the statement $\phi$.

To see that $\phi$ is, however, as-OTM-realizable, note that there are only countably many OTM-computable real numbers (as there are only countably many OTM-programs), so that the set of OTM-incomputable real numbers has measure $1$. For each such real number $x$, the remaining statement that for all $k\in\omega$, all ordinals $\alpha$ and all computations of $P_{k}$ of length $\alpha$, $P_{k}$ does not terminate in $\alpha$ many steps with output $x$ is of the form $\forall{x}\psi$, where $\psi$ is $\Delta_{0}$, so that we can apply Lemma \ref{pi2-truth}.

\end{proof}

We can generalize this argument somewhat:

\begin{definition}
    For a set $x$ of ordinals, $\sigma^{x}$ denotes the first stable ordinal relative to $x$, i.e., the minimal ordinal $\alpha$ such that $L_{\alpha}[x]$ is a $\Sigma_{1}$-submodel of $L[x]$.\footnote{See, e.g., \cite{Barwise}, Theorem $8.2$.}
\end{definition}

We recall some basic information about the function $x\mapsto\sigma^{x}$. 

\begin{lemma}{\label{sigma facts}}
Let $x$ be a set of ordinals. 
\begin{enumerate}
    \item If $x$ is a real number, $\sigma^{x}$ is countable. 
    \item $\sigma^{x}$ is the supremum of the halting times of parameter-free OTMs with oracle $x$. 
    \item $\sigma^{x}$ is the supremum of the ordinals that have OTM-computable codes in the oracle $x$. 
    \item A set of ordinals $y$ is OTM-computable relative to $x$ if and only if $y\in L_{\sigma^{x}}[x]$. 
\end{enumerate}
\end{lemma}
\begin{proof}
    (1) is folklore. (2)-(4) are implicit in the work of Koepke and elaborated, e.g., in \cite{CarlBook}, proof of Lemma 3.5.2.
\end{proof}

\begin{corollary}
Let $\psi$ be the statement ``For each real number $x$, there is a real number $y$ such that $y$ is not OTM-computable from $x$''. Then $\psi$ is not OTM-realizable, but as-OTM-realizable.
\end{corollary}
\begin{proof}
    It follows immediately from Lemma \ref{sigma facts}.4 that $\psi$ is not OTM-realizable, as setting $x=0$ in $\psi$ makes the $\phi$ from the proof of Theorem \ref{difference} a special case of $\psi$. 

    To see that $\psi$ is as-OTM-realizable, note that, by Lemma \ref{sigma facts}.1, $\sigma^{x}$ is a countable ordinal, so that $L_{\sigma^{x}}[x]$ is countable, which means that $\mu(\mathbb{R}\setminus L_{\sigma^{b}}[b])=1$. The rest works as in the proof of Theorem \ref{difference}.
\end{proof}

Assuming $V=L$, we can give an even more natural example, using Corollary $11$ in \cite{CS}, according to which, under $V=L$, the halting problem for OTMs is solvable by an OTM with a random oracle with probability $1$. 

\begin{definition}{\label{halting sets}}
Suppose that $a\subseteq\omega$. 
\begin{itemize}
    \item 
    Define $h_{0}:=a$, $h_{\iota+1}^{a}:=\{i\in\omega:P_{i}^{h_{\iota}^{a}}\downarrow\}$ for $\iota\in\omega$.
We call $h_{\iota}^{a}$ the $\iota$-th iterated OTM-halting problem relative to $a$ or the $\iota$-th OTM-jump of $a$. If $a=\emptyset$, the superscript is dropped.
   \item Define $h^{\text{as},a}:=a$, $h^{\text{as},a}_{n+1}:=\{i\in\omega:\mu(\{x\subseteq\omega:P_{i}^{h_{n}^{\text{as},a},x}\downarrow\})=1\}$ for $n\in\omega$. Thus, in particular $h_{1}^{\text{as},a}$ is the set of OTM-programs that halt relative to $a$ and a randomly chosen oracle $x$ with probability $1$. We call $h_{i}^{a}$ the $i$-th iterated as-OTM-halting problem relative to a.
\end{itemize}
If $a=\emptyset$, the superscript is dropped.
\end{definition}

\begin{definition}
    We say that $b\subseteq\omega$ is as-OTM-computable relative to $a\subseteq\omega$ if and only if there are an OTM-program $P$ and a set $X\in\mathcal{B}$ such that, for all $x\in X$, $P^{x}(a)\downarrow=b$. 
\end{definition}

We note that the jump operator works as expected:

\begin{lemma}{\label{halting not solvable}}
Let $a\subseteq\omega$. 
    \begin{enumerate}
        \item $h_{1}^{a}$ is not OTM-computable relative to $a$.
        \item $h_{1}^{\text{as},a}$ is not as-OTM-computable relative to $a$.
    \end{enumerate}
\end{lemma}
\begin{proof}

(1) is folklore and works exactly as for ordinary Turing machines. 

The proof for (2) is also an easy adaptation of the usual argument, but since the involvement of the oracle set somewhat blurs the waters, we give the details for the sake of the reader. 

So suppose that $H$ as-OTM-computes $h_{1}^{\text{as}}$ (we only deal with the unrelativized case, the relativization being straightforward). Thus, for all $i\in\omega$ there is a set $\mathcal{O}_{1,i}\in\mathcal{B}$ such that, for all $x\in\mathcal{O}_{1,i}$, $H^{x}(i)$ halts with output $1$ if and only if $P_{i}^{y}$ halts for all $y$ from some measure $1$ set, and otherwise with output $0$. 
Since countable intersections of measure $1$ sets are again of measure $1$, we can -- by replacing each $\mathcal{O}_{i,1}$ with the intersection of all these sets, if necessary -- assume without loss of generality that all the $\mathcal{O}_{1,i}$ are all equal to a set $\mathcal{O}_{1}$. Modify $H$ to $\hat{H}$ such that, for $i\in\omega$, some $X\in\mathcal{B}$ and all $x\in X$, $\hat{H}^{x}(i)$ halts if and only if it is not the case that $\mu(\{y\subseteq\omega:P_{i}^{y}(i)\downarrow\})=1$. Suppose that $\hat{H}$ is $P_{k}$ and consider the set $S:=\{x\subseteq\omega:\hat{H}^{x}(k)\downarrow\}$.

If $\mu(S)=1$, then $P_{k}^{y}(k)$ halts for all elements of a measure $1$ set, so, by definition of $\hat{H}$, $\hat{H}^{x}(k)$ does not halt for a measure $1$ set of $x$, a contradiction, since $\hat{H}$ is $P_{k}$. 

If $S$ is not of measure $1$ (i.e., not measurable or of smaller measure), then it is not the case that $P_{k}^{y}(k)$ halts for a measure $1$ set of $y$, so $\hat{H}^{y}(k)$ should halt for a measure $1$ set of $y$, again a contradiction.
\end{proof}

\begin{lemma}{\label{probable halting is sigma2}}
The statement $\phi_{\text{as-halts}}(i)$ stating, for $i\in\omega$, that $\mu(\{x\subseteq\omega:P_{i}^{a,x}\downarrow\})=1$ is $\Sigma_{2}$ in the parameter $a$. 
\end{lemma}
\begin{proof}
$\phi_{\text{as-halts}}(i)$ can be reformulated as ``There is a set $X\subseteq\mathfrak{P}(\omega)$ such that $\mu(X)=0$ and, for all $x\notin X$, we have $P^{a,x}\downarrow$''. A null set can be encoded as a countable sequence of countable sequences of rational intervals, which in turn can be encoded by a real number $z$. The statements that a given real number $z$ codes a null set and, that a given real number $x$ belongs to the coded set are $\Delta_{1}$.
 Thus, the result follows.

\end{proof}

\begin{lemma}{\label{probable halting in L is sigma2-universal}}
    Suppose that $V=L$. Then $h_{1}^{\text{as}}$ is OTM-computable equivalent to $\Sigma_{2}$-truth in $L_{\omega_{1}}$.
\end{lemma}
\begin{proof}
    One direction is immediate from Lemma \ref{probable halting is sigma2}.
    
    For the other reduction, let $\phi:\Leftrightarrow\forall{y}\exists{z}\psi(y,z)$, where $\psi$ is $\Delta_{0}$. Consider the OTM-program $P$ that, in the oracle $x$, works as follows: Enumerate $L$ until $x$ appears, and let $L_{\alpha}$ be the first $L$-level in which it does. The level $L_{\alpha}$ is countable (and coded by a real number), so we can simultaneously run OTM-computations on all $a\in L_{\alpha}$ searching for some $y$ such that $\psi(a,y)$; the computation will halt once such $y$ have been found for all $a\in L_{\alpha}$. We claim that this computation halts for all $x$ from a set of measure $1$ if and only if $\phi$ is true in $L_{\omega_{1}}$. 

    First, if $\phi$ holds in $L_{\omega_{1}}$, then there is an appropriate $y$ for any $a\in L_{\omega_{1}}$, which will eventually be found, causing the computation to halt. 

    On the other hand, suppose that the computation halts for all $x$ from a set of measure $1$. Then there are cofinally in $\omega_{1}$ many $L$-levels such that corresponding witnesses for the existential quantifiers are found for all of their elements. Consequently, such witnesses exist for all elements of $L_{\omega_{1}}$. 
\end{proof}

\begin{lemma}
Suppose that $V=L$. 
\begin{enumerate}
\item For every $a\subseteq\omega$ and every $n\in\omega$, $h_{n}^{x}$ is as-OTM-computable.
\item Suppose that $V=L$. Then the statement $\phi_{h}$, given by ``Every OTM-program halts or does not halt'' is as-OTM-realizable (but not plainly OTM-realizable).\footnote{The proof follows the proof of Corollary $11$ in \cite{CS}.}
\item More generally, for every $n\in\omega$, the statement ``Every OTM-program in the oracle $h_{n}$ halts or does not halt'' is as-OTM-realizable.
\end{enumerate}

\end{lemma}
\begin{proof}
\begin{enumerate}
\item 
Again, it is clear that $\phi_{h}$ is not OTM-realizable, since a realizer would map OTM-programs to pairs $(i,r)$, the first component of which indicates whether or not the given program halts, thus solving the OTM-halting problem on an OTM.

Since $\mu(\mathbb{R}\cap(L\setminus L_{\sigma}))=1$, a randomly chosen real number $x$ will, with probability $1$, satisfy $x\notin L_{\sigma}$. Consider the OTM-program $Q$ that enumerates $L$ until it encounters $x$. If $x\notin L_{\sigma}$, this will have a halting time strictly bigger than $\sigma$. To decide, given an OTM-program $P$, whether or not $P$ halts, run $Q^{x}$ and $P$ in parallel and, if $P$ halts, halt with output $0$, but if $Q^{x}$ halts without $P$ having halted, halt with output $1$.

\item We use the same strategy noting that, with probability $1$, the real number $x$ will be such that $h_{n}$ is contained in the minimal $L$-level that contains $x$. 

\end{enumerate}
\end{proof}

\section{Deduction Rules}

We claim that the axioms and rules of intuitionistic predicate calculus, as found in \cite{CGP}, Definition $17$, are sound for as-OTM-realizability in the following strong sense:

\begin{itemize}
    \item If $\mathcal{A}(A_{1},...,A_{n})$ is an axiom scheme using $n$ propositional variables $A_{1},...,A_{n}$, then there is an OTM-program $P_{\mathcal{A}}$ which, on input $\lceil\phi_{1}\rceil,...,\lceil\phi_{n}\rceil$, computes an as-OTM-realizer for the instance $\mathcal{A}(\phi_{1},...,\phi_{n})$. 
    \item If $\{\mathcal{A}_{1},...,\mathcal{A}_{n}\}\vdash\mathcal{C}$ is a deduction rule $\rho$ using $n$ propositional variables $A_{1},...,A_{n}$, then there is an OTM-program $P_{\rho}$ which, on input $(r_{1},...,r_{n})$ with $r_{i}$ an as-OTM-realizer for $\mathcal{A}_{i}$ for $i\in\{1,...,n\}$, computes an as-OTM-realizer $r_{\mathcal{C}}$ for $\mathcal{C}$.
\end{itemize}

The proof will require some preparation. 
We note that, in terms of measure theory, there is no relevant difference between $\bigoplus_{i=0}^{n-1}x_{i}$ and $(x_{0},...,x_{n-1})$:

\begin{lemma}{\label{coding is harmless}}
    For each $n\in\omega$, the map $d_{n}:[0,1]\rightarrow[0,1]^{n}$ defined by $x\mapsto(x_{0},...,x_{n-1})$ where $x=\bigoplus_{i=0}^{n-1}$ is a metric isomorphism for Lebesgue measure. 
\end{lemma}
\begin{proof}
The proof is a straightforward exercise in measure theory, which we omit here for the sake of brevity. A more detailed argument will appear in a forthcoming paper with Galeotti and Passmann.

\end{proof}

We will thus freely confuse $(x,y)$ and $x\oplus y$ below, and, more generally, $(x_{1},...,x_{k})$ and $\bigoplus_{i=1}^{k}x_{i}$.

In order to guarantee the measurability of the occurring sets, we make a set-theoretical extra assumption. We do not know whether the below results hold in ZFC. 

From now on, we assume projective determinacy (PD)\footnote{Cf., e.g., \cite{Martin}.};
 in particular, this implies that all projective sets are Lebesgue measurable. By Lemma \ref{definability of realizability}, this implies that success sets are always Lebesgue measurable.

\begin{definition}{\label{function product}}
Let $X$ be a set, and let $n\in\omega$. Moreover, for $i\in\{0,...,n\}$, let $f_{i}:X^{i}\rightarrow\mathfrak{P}(X)$.\footnote{Note that we have $X^{0}=\{\emptyset\}$. However, it is easier to directly think of $f_{0}$ as an element of $\mathfrak{P}(X)$, and this is what we will do below.} Then $\tilde{\bigotimes}_{i=0}^{n}f_{i}$ is defined as follows: 

$$\tilde{\bigotimes}_{i=0}^{n}f_{i}=\{(x_{0},...,x_{n}):\forall{i\leq n}x_{i}\in f_{i}(x_{0},...,x_{i-1})\}.$$

We call $\tilde{\bigotimes}_{i=0}^{n}f_{i}$ the \textit{dependent product} of $f_{0},...,f_{n}$. 

\end{definition}

\begin{lemma}{\label{large sections imply largeness}}
\begin{enumerate}
\item Let $A\subseteq[0,1]^{2}$ be measurable, let $X\subseteq[0,1]$ have measure $1$ (as a subset of $[0,1]$) and suppose that, for all $x\in X$, the set $A_{x}:=\{y\in [0,1]:(x,y)\in A\}$ satisfies $\mu(A_{x})=1$ (as a subset of $[0,1]$). Then $\mu(A)=1$.
\item More generally, let $n\in\omega$, $f_{i}:[0,1]^{i}\rightarrow\mathfrak{P}([0,1])$ for $i\in\{0,...,n\}$ such that the following holds for all $k\leq n-1$: There is a $\mathcal{O}_{k}\subseteq[0,1]^{k}$ with $\mu(\mathcal{O}_{k})=1$ and $\mu(f_{k}(x_{0},...,x_{k-1}))=1$ for all $(x_{0},...,x_{k-1})\in\mathcal{O}_{k}$. Then $\mu(\tilde{\bigotimes}_{i=0}^{n}f_{i})=1$. 
\end{enumerate}
\end{lemma}

\begin{lemma}{\label{soundness}}
The rules and axioms of intuitionistic first-order logic as given in \cite{Moschovakis}  are sound for as-OTM-realizability.
\end{lemma}

\begin{proof}
The proof is too long to be presented here. To show how the arguments work, we show here the proofs for one deduction rule and one axiom scheme as examples.

     \textbf{Rule}: If $a$ has no free occurences in $\psi$, then $\{\phi(a)\rightarrow\psi\}\vdash\exists{a}\phi(a)\rightarrow\psi$ 
    
    Let as-OTM-realizers $r_{\forall}\Vdash_{\text{as}}^{\text{OTM}}(\phi(a)\rightarrow\psi)$ and $r_{\exists}\Vdash_{\text{as}}^{\text{OTM}}\exists{a}\phi(a)$ be given. Moreover, let $x\subseteq\omega$ be a real number. Finally, let $\xi$ be a relevant input for $\psi$. 
    
    Our as-OTM-realizer $Q$ is an as-OTM-program that, given this data, decomposes $x$ into $x=\bigoplus_{i=0}^{3}x_{i}$ and then computes $((r_{\forall}^{x_{0}}(r_{\exists}^{x_{1}}(1)))^{x_{2}}(r_{\exists}^{x_{1}}(0)))^{x_{3}}(\xi)$. 

   If $x$ is such that $x_{1}\in\text{succ}_{r_{\exists}}$, we will have that $r_{\exists}^{x_{1}}(0)\Vdash_{\text{as}}^{\text{OTM}}\phi(r_{\exists}^{x_{1}}(1))$. Let $w:=r_{\exists}^{x_{1}}(1)$. If moreover $x_{0}\in\text{succ}_{r_{\forall}}^{w}(1)$, then $r_{1}:=r_{\forall}^{x_{0}}(w)\Vdash_{\text{as}}^{\text{OTM}}\phi(w)\rightarrow\psi$. On the other hand, we have that $r_{2}:=r_{\exists}^{x_{1}}(0)\Vdash_{\text{as}}^{\text{OTM}}\phi(w)$. So, if $x_{2}\in\text{succ}_{r_{1}}^{r_{2}}$, then $r_{3}:=(r_{\forall}^{x_{0}}(r_{\exists}^{x_{1}}(1)))^{x_{2}}(r_{\exists}^{x_{1}}(0))\Vdash_{\text{as}}^{\text{OTM}}\psi$. Thus, finally, if $x_{3}\in\text{succ}_{r_{3}}^{\xi}$, then $r_{3}^{x_{3}}(\xi)$ is as desired. 

Again by Lemma \ref{large sections imply largeness}, since the occuring success sets have measure $1$ by definition of as-OTM-realizability, the set of $x$ satisfying these conditions has measure $1$. 

\bigskip

\textbf{Axiom Scheme}: $(\phi\rightarrow\chi)\rightarrow((\psi\rightarrow\chi)\rightarrow((\phi\vee\psi)\rightarrow\chi))$

Let $r_{\phi\rightarrow\chi}\Vdash_{\text{as}}^{\text{OTM}}\phi\rightarrow\chi$ be given. Our procedure will map this to the OTM-program $r$ which works as follows: Given $r_{\psi\rightarrow\chi}\Vdash_{\text{as}}^{\text{OTM}}\psi\rightarrow\chi$, output the OTM-program $r^{\prime}$ which has the following function: Given $r_{\phi\vee\psi}\Vdash_{\text{as}}^{\text{OTM}}\phi\vee\psi$, along with $x\subseteq\omega$, first compute $j:=r_{\phi\vee\chi}^{x_{0}}(0)$. Then one of the following cases occurs:

\begin{enumerate}
\item If $j=0$, compute $r_{\phi\rightarrow\chi}^{x_{1}}(r_{\phi\vee\psi}^{x_{0}}(1))$ and output it.
\item If $j=1$, compute $r_{\psi\rightarrow\chi}^{x_{1}}(r_{\phi\vee\psi}^{x_{0}}(1))$ and output it.
\end{enumerate}

If $x$ is such that $x_{0}\in\text{succ}_{r_{\phi\vee\psi}}$, then $j\in\{0,1\}$ and $r_{0}:=r_{\phi\vee\psi}^{x_{0}}(1)$ will be an as-OTM-realizer of the $j$-th element of the disjunction $\phi\vee\psi$. If $j=0$, and $x_{1}\in\text{succ}_{r_{\phi\rightarrow\chi}}^{r_{0}}$, then the the output $r_{\phi\rightarrow\chi}^{x_{1}}(r_{0})$ will be an as-OTM-realizer of $\chi$. Similarly if $j=1$ and $x_{1}\in\text{succ}_{r_{\psi\rightarrow\chi}}^{r_{0}}$. Thus, for each $x_{0}\in\text{succ}_{r_{\phi\vee\psi}}$, the set of corresponding $x_{1}$ has measure $1$. By Lemma \ref{large sections imply largeness}, the set of $x$ satisfying these conditions has measure $1$. 
\end{proof}

\section{Axioms of Set Theory}

We now consider as-OTM-realizability of particular set-theoretical statements. It will turn out that all axioms of KP are as-OTM-realizable, while already $\Sigma_{2}$-comprehension is not.
The axiom of replacement and the axiom of choice are as-OTM-realizable in their usual formulation, which, however, is due to the considerable deviation of the realizability interpretation from the classical semantics.

For the axiom schemes of induction, comprehension and collection, we interpret as-OTM-realizabilty in the same (strong) sense that we used for logical axioms when we prove that a certain statement is as-OTM-realizable; for negative results, we provide a concrete instance that fails to be as-OTM-realizable.

The next result, and its proof, are analogous to Theorem $43$ of \cite{CGP}, where this was done for infinitary intuitionistic set theory.

\begin{lemma}
\begin{enumerate}
    \item The axioms of extensionality, empty set, pairing and union are as-OTM-realizable.
    \item The $\Delta_{0}$-separation scheme is as-OTM-realizable.
    \item The (full) collection scheme is as-OTM-realizable.
    \item The induction scheme is OTM-realizable.
\end{enumerate}
\end{lemma}

        Note that, since $x\setminus y$ is easily computable from $x$ and $y$, the as-OTM-realizability of $\Sigma_{n}$-separation is equivalent to that of $\Pi_{n}$-separation. 

\begin{lemma}{\label{comprehension failure}} [Cf. \cite{CarlBook}, Proposition 9.4.4]\footnote{We remark, however, that the argument is different, as the power set operator is not available in the present context.} 
\begin{enumerate}

\item The as-OTM-realizability of the $\Sigma_{1}$-comprehension scheme (and, equivalently, that of the $\Pi_{1}$-comprehension scheme) is independent of ZFC. 
\item (It is provable in ZFC that) the $\Sigma_{2}$-comprehension scheme (and, hence, the $\Sigma_{n}$-comprehension scheme and the $\Pi_{n}$-comprehension scheme for all $n\geq2$) is not as-OTM-realizable.

\end{enumerate}
\end{lemma}
\begin{proof}
\begin{enumerate}
    \item 
    \begin{enumerate}
        \item Suppose first that $V=L$. Let $X\in L$ be a set, $\phi\equiv\exists{y}\psi(x,y)$ a $\Sigma_{1}$-formula. Since $X\in (H_{\omega_{1}})^{L}=L_{\omega_{1}}$, there is a (constructibly) countable ordinal $\alpha$ such that $X\subseteq L_{\alpha}$. Let $\beta:=\text{sup}\{\sigma^{x}:x\in X\}$, then $s<\omega_{1}^{L}$. It follows that (in $L$), we have $\mu(\mathbb{R}\cap (L\setminus L_{\beta})=1$. Thus, if $\alpha(z)$ denotes the minimal $\alpha$ with $z\in L_{\alpha}$, then our randomly chosen real $z$ will satisfy $\alpha(z)>\beta$ with probability $1$.        

 Now, given $x\in X$, use $z$ to calculate $L_{\alpha(z)}$ and search $L_{\alpha(z)}$ for a witness $y$ for the statement $\exists{y}\psi(x,y)$. If such a $y$ is found, $x$ belongs to the specified subset. If no such $y$ is found, then, as $\lambda(r)>\sigma^{x}$, none exists and $x$ does not belong to the specified subset. Thus, the as-OTM-realizability of $\Sigma_{1}$-comprehension is consistent relative to ZFC.
        
        \item On the other hand, let $M$ be a transitive model of ZFC such that, for each $x\subseteq\omega$, the set of real numbers generic for random forcing over $L[x]$ has measure $1$. By Theorem $13$ of \cite{CS}, this implies that, in $M$, the analogue of Sacks' theorem holds for (parameter-free) OTMs: If $x\subseteq\omega$ is OTM-computable relative to all elements of a set of positive measure, then $x$ is OTM-computable ($\ast$). Let $r\Vdash_{\text{as}}^{\text{OTM}}\forall{X}\exists{y}\forall{z}(z\in Y\leftrightarrow(z\in X\wedge \phi(z)))$, where $\phi(z):\Leftrightarrow P_{z}(0)\downarrow$. Take $x:=\omega$. Then there is $\mathcal{O}\in\mathcal{B}$ such that, for all $a\in\mathcal{B}$, we have $r^{a}(\omega)=\{i\in\omega:P_{i}(0)\downarrow\}$, i.e., $r^{a}(\omega)=h_{1}$. By ($\ast$), $h_{1}$ is OTM-computable, a contradiction. Thus, in $M$, the $\Sigma_{1}$-comprehension axiom is not as-OTM-realizable.

    \end{enumerate}

    \item To see that $\Sigma_{2}$-comprehension is not as-OTM-realizable, recall from Lemma \ref{probable halting is sigma2} that the formula $\phi_{\text{as-halts}}(i)$, expressing ``$P_{i}^{x}$ halts with probability $1$'' is $\Sigma_{2}$ and that, by Lemma \ref{halting not solvable} $h_{1}^{\text{as}}=\{i\in\omega:\phi_{\text{as-halts}}(i)\}$ is not as-OTM-computable. Thus, the instance of $\Sigma_{2}$-comprehension where $X$ is $\omega$ and $\phi$ is $\phi_{\text{as-halts}}$ is not as-OTM-realizable.

\end{enumerate}
\end{proof}

Since $H_{\omega_{1}}$ does not even contain $\mathfrak{P}(\omega)$, the power set axiom cannnot be meaningfully considered in the present context.

\begin{lemma}
    The axiom of choice,formulated as $$\forall{X}(\forall{y}(y\in X\rightarrow \exists{z}z\in y)\rightarrow \exists{F:X\mapsto\bigcup{X}}\forall{y\in X}F(y)\in y),$$ is as-OTM-realizable.
\end{lemma}
\begin{proof}
Let $X\in H_{\omega_{1}}$, $r\Vdash_{\text{as}}^{\text{OTM}}\forall{y}(y\in x\rightarrow\exists{z}z\in y)$. Thus, for every $x\in X$, there is a set $\mathcal{O}_{x}\in\mathcal{B}$ such that, for all $a_{0}\oplus a_{1}:=a\in\mathcal{O}_{x}$, we have that $(r^{a_{0}}(x))^{a_{1}}(1)\in x$. Since $X\in H_{\omega_{1}}$, $X$ is countable, and so $\mathcal{O}:=\bigcup_{x\in X}\mathcal{O}_{x}$ satisfies $\mu(\mathcal{O})=1$. 
Now, for all $a\in\mathcal{O}$ and all $x\in X$, we have that $(r^{a_{0}}(x))^{a_{1}}(1)\in x$, so we can let $F(x)=(r^{a_{0}}(x))^{a_{1}}(1)$, which will be as desired.

\end{proof}

\section{Conclusion and further work}

We have shown that almost sure realizability with OTMs has the properties that one would expect a notion of set-theoretical realizability to have, while still differing from ``plain'' OTM-realizability; one could say that, set-theoretically, randomness is informative. Given that the axiom of choice is commonly regarded as highly non-constructive, the fact that the axiom of choice is realizable in this sense is somewhat counter-intuitive. This phenomenon, however, is well-known in constructive set theory. For this reason, a number of formulations and variants of choice principles has been defined (see, e.g., \cite{Rathjen}), and it would be worthwhile to see whether these are as-OTM-realizable. Moreover, there is a number of variants of our notion that warrant consideration.

\newpage

\section{Appendix: Proofs}

We give the proofs of some lemmas and theorems that had to be omitted in the paper due to the page limit for the sake of the referees.

Lemma \ref{pi2-truth}:

\begin{lemma}(Cf. \cite{CGP-OTM}, Lemma $21$-$23$)

There is an OTM-program $P_{\text{eff}-\Pi_{2}}$ that, for a true formula $\phi$ of the form $\forall{x}\exists{y}\psi$, where $\psi$ is $\Delta_{0}$, $P_{\text{eff}-\Pi_{2}}(\lceil\phi\rceil)$ outputs a plain OTM-realizer for $\phi$ which is also an as-OTM-realizer for $\phi$. 

Consequently, the following are equivalent for such $\phi$:
    
\begin{enumerate}
    \item $\phi$ is plainly OTM-realizable. 
    \item $\phi$ is as-OTM-realizable.
    \item $\phi$ is true.
\end{enumerate}

\end{lemma}
\begin{proof}
    This is trivial for the case that $\phi$ itself is $\Delta_{0}$; the details are analogous to those in the proof of \cite{CGP-OTM}, Lemma $23$. We assume from now on that the statement is proved for this case. 

    In the general case, the realizer $r$ will work as follows: Given a real number $x$ coding a set $a$, use the OTM-program $P_{L}$ from \cite{CarlBook}, Lemma $3.5.3$\footnote{Implicit in Koepke \cite{Koepke}, Theorem $2$.} that, on input $x$, enumerates $L[x]$. For each $y$ occuring in the enumeration, check whether $y$ codes a set $b$ such that $\psi(a,b)$ (this is easily possible because $\psi$ only contains bounded quantifiers). Once such $y$ has been found, we compute a realizer $(s,\vec{p})$ for $\psi(a,b)$ and output $(s,(y,\vec{p}))$. 

    In order to see that such $b$ will be found, note that, due to Shoenfield's absoluteness theorem and the fact that, by assumption, $V$ satisfies $\exists{b}\psi(a,b)$, such a set $b$ -- along with a real number coding it -- will exist in $L[x]$.

    The second statement now follows immediately.
\end{proof}

Lemma \ref{coding is harmless}
\begin{lemma}
    For each $n\in\omega$, the map $d_{n}:[0,1]\rightarrow[0,1]^{n}$ defined by $x\mapsto(x_{0},...,x_{n-1})$ where $x=\bigoplus_{i=0}^{n-1}$ is a metric isomorphism for Lebesgue measure. 
\end{lemma}
\begin{proof}

The proof is a straightforward exercise in measure theory, which we omit here for the sake of brevity. A more detailed argument will appear in a forthcoming paper with Lorenzo Galeotti. 

For the sake of the referees, we present here the argument, following the same steps as in version worked out by Galeotti. 

For $s\in\{0,1\}^{<\omega}$, let us denote $|s|$ the length of $s$ and by $I_{s}$ the basic open interval associated with $s$, i.e., the set $\{x\in\{0,1\}^{\omega}:\forall{i<|s|}x(i)=s(i)\}$. 

We first show this for $n=2$, the general case being an obvious generalization. Let $s\in\{0,1\}^{<\omega}$. Then $\mu(I)=2^{-|s|}$. Let $s=s_{0}s_{1}...,s_{k-1}$, and let $s^{0}$, $s^{1}$ be the elements of $\{0,1\}^{<\omega}$ that consist of the bits at even and odd positions of $s$, respectively. In particular, this means that $|s^{0}|+|s^{1}|=|s|$. Then $f[I]=I_{s^{0}}\times I_{s^{1}}$, and so $\mu(f[I])=\mu(I_{s^{0}}\times I_{s^{1}})=\mu(I_{s^{0}})\cdot\mu(I_{s^{1}})=s^{-|s_{0}|}\cdot s^{-|s_{1}|}=s^{-(|s^{0}|+|s^{1}|)}=s^{-|s|}=\mu(I)$. 

Let $\tilde{\mu}$ be the measure on $[0,1]$ defined by $\tilde{\mu}(X)=\mu(f[X])$ for $X\subseteq[0,1]$. Then $\mu$ and $\tilde{\mu}$ are two measure on $[0,1]$ that agree on all basic open sets, and thus on all open sets. By \cite{Bogachev}, Lemma 7.1.2, we have $\tilde{\mu}=\mu$. Thus $\mu(X)=\mu(f[X])$ for all $X\subseteq[0,1]$.

\end{proof}

Lemma \ref{large sections imply largeness}:

\begin{lemma}
\begin{enumerate}
\item Let $A\subseteq[0,1]^{2}$ be measurable, let $X\subseteq[0,1]$ have measure $1$ (as a subset of $[0,1]$) and suppose that, for all $x\in X$, the set $A_{x}:=\{y\in [0,1]:(x,y)\in A\}$ satisfies $\mu(A_{x})=1$ (as a subset of $[0,1]$). Then $\mu(A)=1$.
\item More generally, let $n\in\omega$, $f_{i}:[0,1]^{i}\rightarrow\mathfrak{P}([0,1])$ for $i\in\{0,...,n\}$ such that the following holds for all $k\leq n-1$: There is a $\mathcal{O}_{k}\subseteq[0,1]^{k}$ with $\mu(\mathcal{O}_{k})=1$ and $\mu(f_{k}(x_{0},...,x_{k-1}))=1$ for all $(x_{0},...,x_{k-1})\in\mathcal{O}_{k}$. Then $\mu(\tilde{\bigotimes}_{i=0}^{n}f_{i})=1$. 
\end{enumerate}
\end{lemma}
\begin{proof}
\begin{enumerate}
\item Let $\chi_{A}$ be the characteristic function of $A$; by assumption, $\chi_{A}$ is measurable. Define $f:[0,1]\rightarrow\{0,1\}$ by $$f(a,b)=\begin{cases}\chi_{A}(a,b)\text{, if }a\in X\\1\text{, otherwise}\end{cases}.$$ Clearly, $f$ is measurable as well, as it only differs from $\chi_{A}$ on a set of measure $0$. By Tonelli's theorem, $\int_{[0,1]\times[0,1]}f(x,y)d(x,y)=1$, hence, for $\tilde{A}:=\{(x,y)\in[0,1]^{2}:f(x,y)=1\}$, we have $\mu(\tilde{A})=1$. But $A\subseteq\tilde{A}$ and $\mu(\tilde{A}\setminus A)=0$. Hence $\mu(A)=1$. 

\item This now follows inductively with the help of Lemma \ref{coding is harmless}. For $n=2$, this is part (1). Now suppose that the statement is true for $n$. Encode $\tilde{\bigotimes}_{i=0}^{n}f_{i}$ as a subset of $[0,1]$ as described above.
For each $x\in [0,1]$, let $S_{x}:=\{(x_{1},...,x_{n}):(x,x_{1},...,x_{n})\in\tilde{\bigotimes}_{i=0}^{n}f_{i}\}$. By the inductive assumption, $\mu(S_{x})=1$ for all $x\in f_{0}$. Define, for each $x\in[0,1]$, $\tilde{S}_{x}:=d_{n-1}(S_{x})$, and let $Q:=\{(x,y)\in[0,1]^{2}:y\in S_{x}\}$. Then $\mu(Q)=1$ by the case $n=2$ and thus $\mu(\tilde{\bigotimes}_{i=0}^{n}f_{i})=1$ by Lemma \ref{coding is harmless}.

\end{enumerate}
\end{proof}

Lemma \ref{soundness}:

\begin{lemma}
The rules and axioms of intuitionistic first-order logic as given in \cite{Moschovakis}  are sound for as-OTM-realizability.
\end{lemma}

The proof of this statement is split into the following two sublemmas, which state that soundness holds for the deduction rules and for the axiom (schemes) of the mentioned calculus, respectively. The arguments are similar to those given in the (unpublished) appendix to \cite{CGP}, which will appear in forthcoming work with Galeotti and Passmann; however, the transition to transfinite computations requires some extra care, so that we work the details in most cases.

\begin{lemma}{\label{rule soundness}}
The following rules are sound for as-OTM-realizability:
\begin{enumerate}
    \item $\{\phi,\phi\rightarrow \psi\}\vdash \psi$
    \item If $a$ has no free occurences in $\phi$, then $\{\phi\rightarrow\psi\}\vdash\phi\rightarrow\forall{a}\psi(a)$
    \item If $a$ has no free occurences in $\psi$, then $\{\phi(a)\rightarrow\psi\}\vdash\exists{a}\phi(a)\rightarrow\psi$
\end{enumerate}
\end{lemma}
\begin{proof}

\begin{enumerate} 
    \item Let as-OTM-realizers $r_{\phi},r_{\phi\rightarrow\psi}$ for $\phi$ and $\phi\rightarrow\psi$ be given. 
    The as-OTM-realizer computed from $r_{\phi}$ and $r_{\phi\rightarrow\psi}$ is the program $Q$ that, in the oracle $x$, does the following: Let a relevant input $\xi$ for $\psi$ be given. Decompose $x$ into $x=x_{0}\oplus x_{1}$.
 Compute $r_{x_{0}}:=r_{\phi\rightarrow\psi}^{x_{0}}(r_{\phi})$. Then compute $r_{x_{0}}^{x_{1}}(\xi)$ and return the result. 

    To see that this works, first note that, if $x_{0}\in \text{succ}^{r_{\phi}}_{r_{\phi\rightarrow\psi}}$, then 
        $r_{x_{0}}$ will be an as-OTM-realizer for $\psi$.
    If $x$ is such that $x_{0}\in \text{succ}^{r_{\phi}}_{r_{\phi\rightarrow\psi}}$, 
    and $x_{1}\in\text{succ}^{\xi}_{r_{x_{0}}}$, 
    then $r_{x_{0}}^{x_{1}}(\xi)$ will by definition have the property required by Definition \ref{main def}. But by Lemma \ref{large sections imply largeness} (and Lemma \ref{coding is harmless}), we have that $\mathcal{O}:=\{x_{0}\oplus x_{1}:x_{0}\in\mathcal{O}_{\phi\rightarrow\psi}\wedge x_{1}\in \mathcal{O}_{\psi,x_{0}}\}$ has measure $1$. 
    \item       If $x$ does not occur freely in $\psi$ either, this is trivial. Otherwise, by definition of as-OTM-realizability for formulas with free variables, an as-OTM-realizer for $\phi\rightarrow\psi(a)$ is one for $\forall{a}(\phi\rightarrow\psi(a))$. Let such an as-OTM-realizer $r_{\phi\rightarrow\psi(a)}$ be given. Our as-OTM-realizer for $\phi\rightarrow\forall{a}\psi(x)$ will work in the following way: 

  Given an as-OTM-realizer $r_{\phi}$ for $\phi$, a relevant input $\xi$ for $\forall{y}\psi(y)$ and a real number $x$, first let $x:=x_{0}\oplus x_{1}$. Then compute $(r_{\phi\rightarrow\psi(a)}^{x_{0}}(\xi))^{x_{1}}(r_{\phi})$.

  This works whenever $x$ is such that $x_{0}\in\text{succ}^{\xi}_{r_{\phi\rightarrow\psi(a)}}$ -- which means that $r_{1}:=r_{\phi\rightarrow\psi(a)}^{x_{0}}(\xi)\Vdash_{\text{as}}^{\text{OTM}}\phi\rightarrow\psi(\xi)$ and $x_{1}\in\text{succ}^{r_{\phi}}_{r_{1}}$ -- which means that $r_{1}^{x_{1}}(r_{\phi})\Vdash_{\text{as}}^{\text{OTM}}\psi(\xi)$. By Lemma \ref{large sections imply largeness}, these $x$ form a set of measure $1$.

    \item     Let as-OTM-realizers $r_{\forall}\Vdash_{\text{as}}^{\text{OTM}}(\phi(a)\rightarrow\psi)$ and $r_{\exists}\Vdash_{\text{as}}^{\text{OTM}}\exists{a}\phi(a)$ be given. Moreover, let $x\subseteq\omega$ be a real number. Finally, let $\xi$ be a relevant input for $\psi$. 
    
    Our as-OTM-realizer $Q$ is an as-OTM-program that, given this data, decomposes $x$ into $x=\bigoplus_{i=0}^{3}x_{i}$ and then computes $((r_{\forall}^{x_{0}}(r_{\exists}^{x_{1}}(1)))^{x_{2}}(r_{\exists}^{x_{1}}(0)))^{x_{3}}(\xi)$. 

   If $x$ is such that $x_{1}\in\text{succ}_{r_{\exists}}$, we will have that $r_{\exists}^{x_{1}}(0)\Vdash_{\text{as}}^{\text{OTM}}\phi(r_{\exists}^{x_{1}}(1))$. Let $w:=r_{\exists}^{x_{1}}(1)$. If moreover $x_{0}\in\text{succ}_{r_{\forall}}^{w}(1)$, then $r_{1}:=r_{\forall}^{x_{0}}(w)\Vdash_{\text{as}}^{\text{OTM}}\phi(w)\rightarrow\psi$. On the other hand, we have that $r_{2}:=r_{\exists}^{x_{1}}(0)\Vdash_{\text{as}}^{\text{OTM}}\phi(w)$. So, if $x_{2}\in\text{succ}_{r_{1}}^{r_{2}}$, then $r_{3}:=(r_{\forall}^{x_{0}}(r_{\exists}^{x_{1}}(1)))^{x_{2}}(r_{\exists}^{x_{1}}(0))\Vdash_{\text{as}}^{\text{OTM}}\psi$. Thus, finally, if $x_{3}\in\text{succ}_{r_{3}}^{\xi}$, then $r_{3}^{x_{3}}(\xi)$ is as desired. 

Again by Lemma \ref{large sections imply largeness}, since the occuring success sets have measure $1$ by definition of as-OTM-realizability, the set of $x$ satisfying these conditions has measure $1$. 
\end{enumerate}
\end{proof}

\begin{lemma}{\label{axiom soundness}}
The following axiom schemes are sound for as-OTM-realizability:
\begin{enumerate}
\item $\phi\rightarrow(\psi\rightarrow\phi)$
\item $(\phi\rightarrow\psi)\rightarrow((\phi\rightarrow(\psi\rightarrow\chi))\rightarrow(\phi\rightarrow\chi))$
\item $\phi\rightarrow(\psi\rightarrow(\phi\wedge\psi))$
\item $(\phi\wedge\psi)\rightarrow\phi$ and $(\phi\wedge\psi)\rightarrow\psi$
\item $\phi\rightarrow(\phi\vee\psi)$ and $\psi\rightarrow(\phi\vee\psi)$
\item $(\phi\rightarrow\chi)\rightarrow((\psi\rightarrow\chi)\rightarrow((\phi\vee\psi)\rightarrow\chi))$
\item $(\phi\rightarrow\psi)\rightarrow((\phi\rightarrow\neg\psi)\rightarrow\neg\phi)$
\item $\neg\phi\rightarrow(\phi\rightarrow\psi)$
\item $\forall{x}\phi(x)\rightarrow\phi(t)$
\item $\phi(t)\rightarrow\exists{x}\phi(x)$
\end{enumerate}
\end{lemma}
\begin{proof}
For the most part, the proofs are obvious and the details given in \cite{CGP-OTM}, section 4.1, for ``unrandomized'' OTMs go through. We thus focus on the cases that are more involved or  involve the oracle in some relevant way. 
\begin{enumerate}
    \item Trivial.
\item Let $r_{\phi\rightarrow\psi}\Vdash_{\text{as}}^{\text{OTM}}\phi\rightarrow\psi$. We need to compute an as-OTM-realizer $r$ for $((\phi\rightarrow(\psi\rightarrow\chi))\rightarrow(\phi\rightarrow\chi))$. $r$ will, on input $r_{\phi\rightarrow(\psi\rightarrow\xi)}\Vdash_{\text{as}}^{\text{OTM}}\phi\rightarrow(\psi\rightarrow\chi)$, output the code of an OTM-program $P$ that does the following: Given a real number $x$ and an as-OTM-realizer $r_{\phi}\Vdash_{\text{as}}^{\text{OTM}}\phi$, we let $x=:x_{0}\oplus x_{1}$. 
Then successively compute $r_{0,x}:=r_{\phi\rightarrow\psi}^{x_{0}}(r_{\phi})$, $r_{1,x}:=r_{\phi\rightarrow(\psi\rightarrow\chi)}^{x_{0}}(r_{\phi})$, $r_{2}:=r_{1,x}^{x_{1}}(r_{0,x})$ and output $r_{2}$. 

This will output an as-OTM-realizer of $\chi$ if $x$ is such that $x_{0}\in\text{succ}_{r_{\phi\rightarrow\psi}}^{r_{\phi}}\cap\text{succ}_{r_{\phi\rightarrow(\psi\rightarrow\chi)}}^{r_{\phi}}$ and $x_{1}\in\text{succ}_{r_{1,x}}^{r_{0,x}}$. By Lemma \ref{large sections imply largeness}, these $x$ form a set of measure $1$. 
\item Trivial. 
\item Trivial. 
\item Trivial. 
\item Let $r_{\phi\rightarrow\chi}\Vdash_{\text{as}}^{\text{OTM}}\phi\rightarrow\chi$ be given. Our procedure will map this to the OTM-program $r$ which works as follows: Given $r_{\psi\rightarrow\chi}\Vdash_{\text{as}}^{\text{OTM}}\psi\rightarrow\chi$, output the OTM-program $r^{\prime}$ which has the following function: Given $r_{\phi\vee\psi}\Vdash_{\text{as}}^{\text{OTM}}\phi\vee\psi$, along with $x\subseteq\omega$, first compute $j:=r_{\phi\vee\chi}^{x_{0}}(0)$. Then one of the following cases occurs:

\begin{enumerate}
\item If $j=0$, compute $r_{\phi\rightarrow\chi}^{x_{1}}(r_{\phi\vee\psi}^{x_{0}}(1))$ and output it.
\item If $j=1$, compute $r_{\psi\rightarrow\chi}^{x_{1}}(r_{\phi\vee\psi}^{x_{0}}(1))$ and output it.
\end{enumerate}

If $x$ is such that $x_{0}\in\text{succ}_{r_{\phi\vee\psi}}$, then $j\in\{0,1\}$ and $r_{0}:=r_{\phi\vee\psi}^{x_{0}}(1)$ will be an as-OTM-realizer of the $j$-th element of the disjunction $\phi\vee\psi$. If $j=0$, and $x_{1}\in\text{succ}_{r_{\phi\rightarrow\chi}}^{r_{0}}$, then the the output $r_{\phi\rightarrow\chi}^{x_{1}}(r_{0})$ will be an as-OTM-realizer of $\chi$. Similarly if $j=1$ and $x_{1}\in\text{succ}_{r_{\psi\rightarrow\chi}}^{r_{0}}$. Thus, for each $x_{0}\in\text{succ}_{r_{\phi\vee\psi}}$, the set of corresponding $x_{1}$ has measure $1$. By Lemma \ref{large sections imply largeness}, the set of $x$ satisfying these conditions has measure $1$. 

\item Recall that we interpret $\neg\phi$ as $\phi\rightarrow(1=0)$. Let $r_{\phi\rightarrow\psi}\Vdash_{\text{as}}^{\text{OTM}}\phi\rightarrow\psi$. We will compute on this input the OTM-program $r$ that works as follows: Given $r_{\phi\rightarrow\neg\psi}\Vdash_{\text{as}}^{\text{OTM}}$, $r_{\phi}\Vdash_{\text{as}}^{\text{OTM}}\phi$ and $x\subseteq\omega$, we proceed as follows: Decompose $x=:x_{0}\oplus x_{1}$. Then compute $r_{0}:=r_{\phi\rightarrow\psi}^{x_{0}}(r_{\phi})$ and $r_{1}:=r_{\phi\rightarrow\neg\psi}^{x_{0}}(r_{\phi})$. Then return $r_{1}^{x_{1}}(r_{0})$. 
If $x$ is such that $x_{0}\in\text{succ}_{r_{\phi\rightarrow\psi}}^{r_{\phi}}\cap \text{succ}_{r_{\phi\rightarrow\neg\psi}}^{r_{\phi}}$ (both of which have measure $1$ by assumption), then $r_{0}\Vdash_{\text{as}}^{\text{OTM}}\psi$ while $r_{1}\Vdash_{\text{as}}^{\text{OTM}}\psi\rightarrow(1=0)$. Thus, if also $x_{1}\in\text{succ}_{r_{1}}^{r_{0}}$, then $r_{1}^{x_{1}}(r_{0})\Vdash_{\text{as}}^{\text{OTM}}(1=0)$, as desired. Again, the set of these $x$ has measure $1$ by Lemma \ref{large sections imply largeness}.
\item Let $r_{\neg\phi}$ be an as-OTM-realizer of $\neg\phi$, i.e., of $\phi\rightarrow(1=0)$. 
If there was an as-OTM-realizer $r$ for $\phi$, then, taking $x\in\text{succ}^{r}_{r_{\neg\phi}}$, we would have that $r_{\neg\phi}^{x}(r_{\phi})\vdash_{\text{as}}^{\text{OTM}}(1=0)$, which, by definition of as-OTM-realizability for atomic formulas, cannot be. It follows that $\phi$ has no as-OTM-realizers. Thus, we can output the OTM-program that returns $\emptyset$ on every input as our as-OTM-realizer for $\phi\rightarrow\psi$, as the requirement for such a realizer are vacuously satisfied.
\item Given $r_{\forall}\Vdash_{\text{as}}^{\text{OTM}}\forall{a}\phi(a)$, a set $t\in H_{\omega_{1}}$ and a real number $x$, we output $r_{\forall}^{x}(t)$. Call this procedure $P$; then $\text{succ}^{t}_{P}=\text{succ}^{t}_{r_{\forall}}\in\mathcal{B}$.
\item Given an as-OTM-realizer $r_{\phi}$ of $\phi(t)$, return the program that, on input $0$, returns $r_{\phi}$ while on input $1$, it returns $t$ (recall that, by clause (4) in the definition of as-OTM-realizers, the parameters used in the antecedent of an implication are passed on to the program realizing the succedent, so that $t$ is indeed available).
\end{enumerate}
    
\end{proof}

\begin{lemma}
\begin{enumerate}
    \item The axioms of extensionality, empty set, pairing and union are as-OTM-realizable.
    \item The $\Delta_{0}$-separation scheme is as-OTM-realizable.
    \item The (full) collection scheme is as-OTM-realizable.
    \item The induction scheme is OTM-realizable.
\end{enumerate}
\end{lemma}
\begin{proof}
\begin{enumerate}
    \item Trivial.
    \item Let a set $X\in H_{\omega_{1}}$ be given, along with a $\Delta_{0}$-formula $\phi$ and a real number $x$.\footnote{Here and below, we suppress mentioning of parameters, as these do not lead to any extra complications.} To compute $X_{\phi}:=\{a\in X:\phi(a)\}$, we can ignore $x$ and use the fact that truth of $\Delta_{0}$-formulas is easily seen to be decidable on an OTM; computing an as-OTM-realizer for a true $\Delta_{0}$-statement $\phi$ (uniformly in $\phi$) then works as in \cite{CGP}, Theorem $25$.
 
    \item     Let $\phi(x,y)$ be $\Delta_{0}$. We want to realize the following statement: 
    $$\forall{X}((\forall{x\in X}\exists{y}\phi(x,y))\rightarrow\exists{Y}\forall{x\in X}\exists{y\in Y}\phi(x,y)).$$ 
    So let a (code for a) set $X$ be given. 
    Moreover, suppose that $r\Vdash_{\text{as-OTM}}\forall{x\in X}\exists{y}\phi(x,y)$. 
    For each $x\in X$, there is a set $\mathcal{O}_{x}\in\mathcal{B}$ such that $\phi(x,r^{y}(x))$ for all $y\in\mathcal{O}_{x}$. Let $\mathcal{O}:=\bigcap_{x\in X}\mathcal{O}_{x}$. Since $X$ is countable, we have $\mu(\mathcal{O})=1$. So, relative to the given oracle $z$, we compute $Y:=\{r^{z}(x):x\in X\}$ as our witness; to realize $\forall{x\in X}\exists{y\in Y}\phi(x,y)$, we take the program $P$ that, given $x\in X$, computes $r^{z}(x)$ (note that $z$ is available by the definition of as-OTM-realizability for implications). 
The definition of $Y$ guarantees that the witnesses thus obtained will be elements of $Y$. This will work whenever $z\in\mathcal{O}$.
  
    \item This leaves us with induction; our argument will be similar to the one for Theorem $43$ in \cite{CGP}. However, showing that the ocuring success sets have measure $1$ requires some care; we therefore give the proof in details.

    We need to provide an OTM-effective method that, given an $\in$-formula $\phi$, computes an as-OTM-realizer for $$\text{Ind}_{\phi}:=\forall{a}((\forall{b\in a}\phi(b))\rightarrow\phi(a))\rightarrow\forall{c}\phi(c).$$ 

    Let an as-OTM-realizer $r_{0}$ for $\forall{a}((\forall{b\in a}\phi(b))\rightarrow\phi(a))$ be given, along with a set $c$ and a real number $x=:x_{0}\oplus x_{1}$. 
    We need to compute an as-OTM-realizer for $\phi(c)$. We start by forming the transitive closure $\text{tc}(\{c\})$ of $\{c\}$. Then, by $\in$-recursion on $\text{tc}(\{c\})$, we will compute, for each $d\in\text{tc}(\{c\})$, two OTM-programs $r_{d}$ and $\tilde{r}_{d}$, where $\tilde{r}_{d}\Vdash_{\text{as}}^{\text{OTM}}\phi(d)$ and $r_{d}$ computes the function mapping each $e\in\text{tc}(d)$ to $\tilde{r}_{e}$. The procedure is perhaps best understood by thinking of $\text{tc}(\{c\})$ as a tree, where the edges are given by the $\in$-relation. In each step of the procedure, $r_{d}$ and $\tilde{r}_{d}$ will be computed for some $d\in\text{tc}(\{c\})$ (we can imagine that these programs are written on reserved portions of the tape next to $d$), after which $d$ will be ``marked'' as already considered. In the beginning, all nodes are unmarked. Then, the following is done as long as there are still unmarked elements in $\text{tc}(\{c\})$:

    \begin{enumerate}
        \item Find an $\in$-minimal unmarked element $d$ of $\text{tc}(\{c\})$.\footnote{Thus, the procedure will always start by considering the empty set. Note that every (safe) OTM-program is an as-OTM-realizer for $\forall{b\in\emptyset}\phi(b)$ by definition, since, by definition, $b\in\emptyset$ has no as-OTM-realizers.} 
        \item That $d$ is $\in$-minimal means that all of its elements $e$ have been marked, so that $\tilde{r}_{e}$ will already have been computed for all $e\in d$. Define $r_{d}$ to be the as-OTM-algorithm that, given $d$ and $e$, follows the steps of the computation made so far to compute $\tilde{r}_{e}$. 
        \item Compute $\tilde{r}_{d}:=(r_{0}^{x_{0}}(d))^{x_{1}}(r_{d})$ and mark $d$ as considered.
    \end{enumerate}

    Once this procedure has finished, $c$ will be marked. This means that $\tilde{r}_{c}$ has been computed, which will be output as the desired as-OTM-realizer of $\phi(c)$.

    We will now prove by $\in$-induction that this procedure works as desired; more precisely, we claim that, for every $c\in H_{\omega_{1}}$, there is a set $\mathcal{O}_{c}\in\mathcal{B}$ such that the above procedure, when run in an oracle $x\in\mathcal{O}_{c}$, produces appropriate $r_{a}$ and $\tilde{r}_{a}$, for every $a\in\text{tc}(\{c\})$. 

 So let $c\in H_{\omega_{1}}$ be given and, by induction, assume the claim proved for all elements of $\text{tc}(c)$, and let $\{\mathcal{O}_{a}:a\in\text{tc}(c)\}$ be the corresponding sets guaranteed to exist by the inductive assumption. Let $\mathcal{O}_{<c}:=\bigcap_{a\in\text{tc}(c)}\mathcal{O}_{a}$. Then, as an intersection of countably many sets of measure $1$, $\mu(\mathcal{O}_{<c})=1$.

We assume from now on that $x\in\mathcal{O}_{<c}$. Then, in the oracle $x$, the above routine computes $r_{a}$ and $\tilde{r}_{a}$ uniformly in $a$, for all $a\in\text{tc}(c)$. Hence this routine, restricted to elements of $\text{tc}(c)$, is already our $r_{c}$.

Suppose that $x$ is such that $x_{0}\in\text{succ}^{c}_{r_{0}}$. Then $r_{1}:=r_{0}^{x_{0}}(c)\Vdash_{\text{as}}^{\text{OTM}}\forall{b\in c}\phi(b)\rightarrow\phi(c)$. By definition, $r_{c}\Vdash_{\text{as}}^{\text{OTM}}\forall{b\in c}\phi(b)$ (since it computes an as-OTM-realizer of $\phi(b)$ for every $b\in \text{tc}(c)\subseteq c$, relative to all $x\in\mathcal{O}_{<c}$).  Thus, if additionally $x_{1}\in \text{succ}_{r_{1}}^{r_{c}}$, then $r_{1}^{x_{1}}(r_{c})\Vdash_{\text{as}}^{\text{OTM}}\phi(c)$.

Thus, on input $c$, the above procedure works as desired if $x\in \mathcal{O}_{<c}$, $x_{0}\in\text{succ}^{c}_{r_{0}}$ and $x_{1}\in \text{succ}_{r_{1}}^{r_{c}}$. By Lemma \ref{large sections imply largeness}, the set $\mathcal{O}$ of those $x$ that satisfy the latter two conditions has measure $1$; thus, the same is true of $\mathcal{O}_{<c}\cap\mathcal{O}$, which is the desired set of measure $1$ such that the above procedure will work relative to all elements of it.

\end{enumerate}
\end{proof}


\begin{thebibliography}{}
\bibitem{Barwise} J. Barwise. Admissible Sets and Structures. Perspectives in Logic. Cambridge University Press. (2017) \url{10.1017/9781316717196}
\bibitem{Bogachev} V. Bogachev. Measure Theory. Volume 1. Physica-Verlag (2007) 
\bibitem{CarlBook} M. Carl.  Ordinal Computability. An Introduction to Infinitary Machines. De Gruyter (2019) 
\bibitem{CarlNote} M. Carl. A Note on OTM-Realizability and Constructive Set Theories. Preprint. arXiv:1903.08945 (2019)
\bibitem{CS} M. Carl, P. Schlicht. Infinite Computations with Random Oracles. Notre Dame Journal of Formal Logic, 58(2) (2017)
\bibitem{DH} R. Downey, D. Hirschfeldt. Algorithmic Randomness and Complexity. Springer (2010)
\bibitem{CGP-OTM} M. Carl, L. Galeotti, R. Passmann. Realisability for Infinitary Intuitionistic Set Theory. Annals of Pure and Applied Logic, 174(6)
(2023)
\bibitem{CGP} M. Carl, L. Galeotti, R. Passmann. Randomizing Realizability. In: L. De Mol et al. (eds) Connecting with Computability. CiE 2021. Lecture Notes in Computer Science, 12813. Springer, Cham. \url{10.1007/978-3-030-80049-9_8}
\bibitem{Hodges} W. Hodges. On the Effectivity of some Field Constructions. Proceedings of the London Mathematical Society, s3--32(1) (1976)
\bibitem{IP} R. Iemhoff, R. Passmann. Logics of intuitionistic Kripke-Platek set theory. Annals of Pure and Applied Logic, 172(10) (2021)
\bibitem{Koepke} P. Koepke. Turing Computations on Ordinals. Bulletin of Symbolic Logic 11(3) (2005) \url{https://doi.org/10.2178/bsl/1122038993}
\bibitem{Lubarsky} B. Lubarsky. IKP and friends. Journal of Symbolic Logic 67(4) (2002) \url{10.2178/jsl/1190150286}
\bibitem{Moschovakis} J. Moschovakis. Intuitionistic Logic. In: E. Zalta, U. Nodelman (eds.). The Stanford Encyclopedia of Philosophy. \url{https://plato.stanford.edu/archives/sum2023/entries/logic-intuitionistic/} (last accessed 05.02.2024) (2023)

\bibitem{Martin} D. Martin, J. Steel. A Proof of Projective Determinacy. Journal of the American Mathematical Society, vol. 2(1). (1989)
\bibitem{Rathjen} M. Rathjen. Choice principles in constructive and classical set theories.  In: Chatzidakis Z, Koepke P, Pohlers W, eds. Logic Colloquium '02. Lecture Notes in Logic. Cambridge University Press. (2006)
\bibitem{Wontner} N. Wontner. Views from a Peak. Generalisations and descriptive set theory. ILLC Dissertation Series DS-2023-NN (2023)
\end{thebibliography}
\end{document}